\documentclass[reqno,11pt,a4paper]{amsart}

\usepackage{amsfonts,amsmath,amssymb,amsthm,color}

\usepackage{amssymb}
\usepackage{amsfonts}
\usepackage{amsmath}
\usepackage{amsthm}
\usepackage{nccmath}
\usepackage{mathrsfs}
\usepackage[USenglish]{babel}
\usepackage{esint}
\usepackage{graphicx}
\usepackage[colorlinks=true]{hyperref}
\usepackage[latin1]{inputenc}
\usepackage[hmargin=3cm, vmargin=2.9cm]{geometry}
\hypersetup{
    colorlinks,
    citecolor=red,
    linkcolor=blue,
    urlcolor=black
}

\allowdisplaybreaks

\newtheorem{theorem}{Theorem}[section]
\newtheorem{lemma}[theorem]{Lemma}
\newtheorem{proposition}[theorem]{Proposition}

\newtheorem{remark}[theorem]{Remark}

\numberwithin{equation}{section}

\renewcommand{\(}{\left(}
\renewcommand{\)}{\right)}

  \newcommand{\e}{\epsilon}
\newcommand{\de}{\delta}

 \newcommand{\intO}{\int_{\Omega}}

\begin{document}

\title[Nodal solutions of the Brezis-Nirenberg problem in dimension  6]{Nodal solutions of the Brezis-Nirenberg problem in dimension  6}
 
\author{Angela Pistoia}
\address[A.Pistoia]{Dipartimento di Scienze di Base e Applicate per l'Ingegneria, Sapienza Universit\`a di Roma, Via Scarpa 16, 00161 Roma, Italy}
\email{angela.pistoia@uniroma1.it}

\author{Giusi Vaira}
\address[G.Vaira]
{Dipartimento di Matematica, Universit\`a degli studi di Bari ``Aldo Moro'', via Edoardo Orabona 4,70125 Bari, Italy}
\email{giusi.vaira@uniba.it}

\begin{abstract}
We show that the classical Brezis-Nirenberg problem
$$
-\Delta u=u|u| + \lambda  u\ \hbox{in}\ \Omega,\\
 u=0\ \hbox{on}\ \partial\Omega,
 $$
when $\Omega$ is a  bounded domain in $\mathbb R^6$ has a sign-changing  solution which blows-up at a point in  $\Omega$ as $\lambda$ approaches a suitable value $\lambda_0>0.$
\end{abstract}

\date\today
 \thanks{ A. Pistoia was partially supported by Fondi di Ateneo ``Sapienza" Universit\`a di Roma (Italy). G. Vaira was partially supported by PRIN 2017JPCAPN003 ``Qualitative and quantitative aspects of nonlinear PDEs" }

\maketitle
 \section{Introduction}
Brezis and Nirenberg 
in their famous paper \cite{Brezis1983} introduced the problem
 \begin{equation}\label{p}
 -\Delta u=|u|^{4\over n-2} u+\lambda u \ \hbox{in}\ \Omega,\ u=0\ \hbox{on}\ \partial\Omega,
 \end{equation}
where $\Omega$ is a bounded   domain  in $\mathbb R^n$ and $n\ge3.$
A huge number of results concerning \eqref{p} has been obtained since then. Let us summarize the most relevant results which are also connected with the topic of the present paper.\\

First of all, the classical Pohozaev's identity ensures that \eqref{p} does not have any solutions if $\lambda\le0$ and  $\Omega$ is a star-shaped domain.
A simple argument shows that problem \eqref{p} does not have any positive solutions if $\lambda\ge \lambda_1(\Omega)$, where $\lambda_1(\Omega)$ is the first eigenvalue of $-\Delta$ with Dirichlet boundary condition. The existence of a least energy positive solution $u_\lambda$ to \eqref{p}, i.e. a solution which achieves 
the infimum
$$m_{ \lambda}:=\inf\limits_{u \in H^1_0(\Omega )\setminus\{0\}}
{\int\limits_{\Omega_\theta} \( |\nabla u |^2-\lambda u ^2\)d  x\over ( \int\limits_{\Omega }  | u |^{p+1}d  x)^{2\over p+1}}$$
  has been proved by Brezis and Nirenberg in \cite{Brezis1983} when  $\lambda\in (0,\lambda_1(\Omega))$  in dimension $n\ge4$ and when $\lambda\in(\lambda^*(\Omega),\lambda_1(\Omega))$ in dimension $n=3$ where $\lambda^*(\Omega)>0$ depends on the domain $\Omega$. If $\Omega$ is the ball then $\lambda^*(\Omega)=\frac14\lambda_1(\Omega) $ (see \cite{Brezis1983}),  while the general case has been treated by Druet in \cite{druet}.
The existence of a sign-changing solution has been proved by Cerami, Solimini and Struwe in \cite{css} when $\lambda\in(0,\lambda_1(\Omega))$ and $n\ge6$ and by Capozzi, Fortunato and Palmieri in \cite{cfp} when $\lambda\ge\lambda_1(\Omega)$ and $n\ge4.$
\\

There is a wide literature about the study of the asymptotic profile of the solutions when the parameter $\lambda$ approaches either zero or some strictly positive values depending on the dimension $n$ and  the domain $\Omega.$ 
In the following, we will focus on the existence of solutions which exhibite a positive or negative blow-up   phenomenon as $\lambda$ approaches some particular values.

When  the parameter $\lambda$ approaches zero, positive and sign-changing solutions which blow-up positively or negatively at one or more points in $\Omega$ do exist provided the dimension $n\ge4.$ 
Rey in \cite{rey},  Musso and Pistoia in \cite{mupi} and Esposito, Pistoia and V\'{e}tois in \cite{epv} built solutions to \eqref{p} with simple positive or negative blow-up  points, i.e. around each point the solution looks like
a positive or a negative standard bubble. Here the standard bubbles are the functions
 \begin{equation}\label{bub} U_{\delta, \xi}(x):= \alpha_n\frac{\delta^{n-2\over2} }{\(\delta^2+|y|^2\)^{n-2\over2}},\  \hbox{with}\ \delta>0,\ \xi\in \mathbb R^n\ \hbox{and}\  \alpha_n:=(n(n-2))^{n-2\over4 },\end{equation} 
 which  are the only
positive solutions of the equation    $-\Delta U= U^{n+2\over n-2}$ in $  \mathbb R^n $ (see \cite{A, cgs, T})
More precisely,  if  $\lambda$ is small enough problem \eqref{p} has  a positive solution which blows-up at one point  (see \cite{rey} if $n\ge5$ and \cite{epv} if $n=4$)
	and   a sign-changing solution which  blows-up positively and negatively at two different points (see \cite{mupi} if $n\ge5$). As far as we know the existence of multiple concentration in the case $n=4$ is still open. 
	If $n=3$ positive solutions of \eqref{p} blowing-up at a single point  when the parameter $\lambda$ approaches a strictly positive number have been found by Del Pino, Dolbeault and Musso in \cite{ddm}.  Moreover, sign-changing solutions having both positive and negative blow-up   points can be constructed arguing as  Musso and Salazar in \cite{ms}, where they found solutions which blow-up at more points when $\lambda$ is close to a suitable strictly positive number. 
 In higher dimension $n\ge7$ 
Premoselli \cite{pre} found  an arbitrary large number of sign-changing solutions to \eqref{p} with a towering blow-up   point in $\Omega$, i.e. around the point the solution likes look like the superposition of
  bubbles of alternating sign  (see also Iacopetti and Vaira \cite{iava1}). In particular, if $\Omega$ is a ball these solutions are nothing but the radially symmetric  nodal solutions.
Conversely, if  $\Omega$ is the ball in low dimension $n=3,4,5,6$,   Atkinson, Brezis and Peletier in \cite{abp} proved that  
problem \eqref{p} does not have any sign-changing radial solutions  when $\lambda\in(0,\lambda_*)$ where $\lambda_*$ depends on the dimension $n.$ In particular, we expect that in low dimension   the blowing-up (or blowing-down) phenomenon takes places when $\lambda $ approaches a positive value different from zero. In fact 
Iacopetti and Vaira in \cite{iava2} proved that if $n=4,5$ and $\lambda$ approaches the first eigenvalue $\lambda_1(\Omega)$ the problem \eqref{p} has a sign-changing solution which  blows-down  at the origin  and   shares the shape of    the positive first eigenfunction associated with $\lambda_1(\Omega)$ far away. So a natural question arises:{\it  is it possible to find a sign-changing blowing-up  solution of \eqref{p}   in dimension  $n=6$ when $\lambda$ approaches some strictly positive number?}
\\

In the present paper, we give a positive answer. In order to state our result, we need to introduce some notation and the assumptions.
\\
Let $u_0$ be a   solution to
\begin{equation}\label{20}
\begin{cases}
-\Delta u_0=|u_0|u_0+ \lambda_0u_0\ \hbox{in}\ \Omega
,\\
 u_0=0\ \hbox{on}\ \partial\Omega
 \end{cases}
 \end{equation}
If $\xi_0\in\Omega$ is such that $\max_\Omega u_0=u_0(\xi_0)>0,$
we   suppose that
\begin{equation}\label{200}
\boxed{\lambda_0=2u_0(\xi_0)} \end{equation}
 We   assume that  $u_0$ is non-degenerate, i.e.
\begin{equation}\label{30}  \boxed{
\begin{cases}
-\Delta v =(2|u_0| + \lambda_0)v \ \hbox{in}\ \Omega\\
 v=0\ \hbox{on}\ \partial\Omega.
 \end{cases}\ \Rightarrow\ v\equiv0}
 \end{equation}
If $v_0$ solves
\begin{equation}\label{v0}
\begin{cases}
-\Delta v_0-(2 |u_0|+\lambda_0 )v_0=u_0\quad\mbox{in}\,\, \Omega\\
v_0=0\ \hbox{on}\ \partial\Omega,
 \end{cases}
 \end{equation}
 we   require that
\begin{equation}\label{v00}\boxed{2v_0(\xi_0) \not=1}\end{equation}
 We will show that the problem
\begin{equation}\label{p1}
\begin{cases}
-\Delta u=u|u| +(\lambda_0+\epsilon) u\ \hbox{in}\ \Omega,\\
 u=0\ \hbox{on}\ \partial\Omega,
 \end{cases}
 \end{equation}
where $\Omega$ is a bounded domain in $\mathbb R^6$ has a sign-changing  solution which blows-down at $\xi_0$ as $|\epsilon|$ approaches zero (note that $\epsilon$ is not necessarily positive) and shares the shape of $u_0$ far away from $\xi_0.$
More precisely our existence result reads as follows.
\begin{theorem}\label{main1} Assume \eqref{200}, \eqref{30} and \eqref{v00}. There exists $\varepsilon_0>0$ such that 
\begin{enumerate}
\item if $1-2v_0(\xi_0) >0$   and  $\epsilon\in(0,\varepsilon_0)$  
\item if $1-2v_0(\xi_0) <0$  and $\epsilon\in(- \varepsilon_0,0)$  
\end{enumerate}
then  there exists a sign-changing solution $u_\epsilon$ of the  problem \eqref{p1}
 which blows-up at the point $\xi_0$ as $\epsilon\to0$. More precisely
$$
u_\epsilon (x)=u_0(x)+\epsilon v_0(x)-PU_{\delta_\epsilon,\xi_\epsilon}(x)+\phi_\epsilon(x)
$$
with as $\epsilon\to0$ 
$$\delta_\epsilon|\epsilon|^{-1}\to d>0,\ \xi_\epsilon\to\xi_0\ \hbox{and}\ \|\phi_\epsilon\|_{H^1_0(\Omega)}=\mathcal O\(\epsilon^2|\ln|\epsilon||^{\frac23}\).$$
\end{theorem}
Here    $P U_{\delta,\xi}$  denotes the projection     onto $H^1_0(\Omega) $  of  the standard bubble  $ U_{\delta, \xi}$  defined in \eqref{bub}, i.e. 
$
-\Delta P U_{\delta,\xi}=U^2_{\delta,\xi}$ in $\Omega$ with 
 $P U_{\delta,\xi}=0$ on $\partial\Omega.$\\
 
It is natural to ask to for which domains $\Omega$ the assumptions \eqref{200}, \eqref{30} and \eqref{v00}
hold true.
 If $\Omega$ is the ball and $u_0$ is the positive solution they are all satisfied  (see \cite{Sri93} for \eqref{30}, \cite{aggpv} for \eqref{200} and \eqref{v00}).
More in general, we can only prove   that assumptions \eqref{200} and \eqref{30} are satisfied for most domains $\Omega$ (see Theorem \ref{main}) when $u_0$ is the least energy positive solution to \eqref{20}. It would be interesting to prove    that \eqref{v00}  also holds for generic domains. \\

Let us state our generic result.
Let $\Omega_0$ be a bounded and smooth domain in $\mathbb R^6$ and let $D$ be an open neighbourhood of $\overline{\Omega_0}$.
Set $\Omega_\theta:=\Theta(\Omega_0)$ where   $\Theta=I+\theta,$  $\theta\in C^{3,\alpha}(\overline D,\mathbb R^6)$ with $ \|\theta\|_{2,\alpha}\le\rho, $ with
$\alpha\in(0,1)$ and $\rho$   small enough. Let us consider the problem on the perturbed domain $\Omega_\theta$
\begin{equation}\label{1theta}
\Delta u+\lambda u+|u| u=0\ \hbox{in}\ \Omega_\theta,\ u=0\ \hbox{on}\ \partial\Omega_\theta.
\end{equation}
\begin{theorem}\label{main-generico}
The set
$$\begin{aligned}\Xi:=\big\{\theta\in C^{3,\alpha}(\overline D,\mathbb R^6)\ :\ &\hbox{if $\lambda>0$ and $u\in H^1_0(\Omega_\theta)$ solves \eqref{1theta} }\\ 
& \hbox{then $u$ is non-degenerate} \big\}\end{aligned}$$
is a residual subset in $C^{3,\alpha}(\overline D,\mathbb R^6),$
 i.e. $C^{3,\alpha}(\overline D,\mathbb R^6)\setminus \Xi$ is a countable union of close subsets without interior points.
\\
Moreover, if $\lambda\in(0,\lambda_1(\Omega_\theta))$ and $u_\lambda$ denotes the least energy positive solution of \eqref{1theta},  for any  $\theta\in \Xi$  there exists $\lambda_\theta\in (0,\lambda_1(\Omega_\theta))$ such that
$$\lambda_\theta=2\max\limits_{\Omega_\theta} u_{\lambda_\theta}.$$
\end{theorem} 

The proof of Theorem \ref{main1} is based upon the well-known Ljapunov-Schmidt reduction. In Section 2 we describe the main steps of the proof by omitting many details which can be found up to minor changes  in the quoted papers. We only prove which cannot be immediately deduced by known results. In particular, we point out the careful construction of the ansatz \eqref{sol}  which has to be  refined up to a second order   and the  delicate estimate of the reduced energy  \eqref{cruc1} given in Proposition \ref{cruc0}   whose leading term   \eqref{cruc2} arises from  the interaction between the bubble and the second order term in the ansatz.\\
The proof of Theorem \ref{main-generico} relies on a classical transversality argument and it is carried out in Section 3.

 \section{The existence of a sign-changing solution}
\subsection{Setting of the problem and the choice of the ansatz}
In what follows we denote by $$(u , v):=\int_\Omega \nabla u\nabla v\, dx,\  \|u\|:=\(\int_\Omega |\nabla u|^2\, dx\)^{\frac 12}\ \hbox{and}\ |u|_r:=\(\int_\Omega |u|^r\, dx\)^{\frac 1 r}$$
 the inner product and its correspond norm in $H^1_0(\Omega)$ and the  standard norm in $L^r(\Omega)$, respectively. When $A\neq \Omega$ is any Lebesgue measurable set we specify the domain of integration by using $\|u\|_A, |u|_{r, A}$. \\ 
Let $(-\Delta)^{-1}: L^{\frac 32}(\Omega)\to H^1_0(\Omega)$ be the operator defined as the unique solution of the equation $$-\Delta u =v \quad \mbox{in}\,\, \Omega\qquad u=0\quad\mbox{on}\,\, \partial\Omega.$$ By the Holder inequality it follows that 
$$\|(-\Delta)^{-1}(v)\|\le C |v|_{\frac 32}\qquad\forall\,\, v\in L^{\frac 32}(\Omega)$$ for some positive constant $C$, which does not depend on $v.$ Hence we can rewrite problem \eqref{p1} as \begin{equation}\label{pb1r} u=(-\Delta)^{-1}[f(u)+(\lambda_0+\epsilon) u]\quad u\in H^1_0(\Omega)\end{equation} with $f(u)=|u|u$.\\

Next we remind the expansion of the projection of the bubble. We denote by $G(x, y)$ the Green's function of the Laplace operator given by 
$$
G(x, y)=\frac{1}{4\omega_6}\(\frac{1}{|x-y|^4}-H(x, y)\)$$
where $\omega_6$ denotes the surface area of the unit sphere in $\mathbb R$ and $H$ is the regular part of the Green's function, namely for all $y\in\Omega$, $H(x, y)$ satisfies $$\Delta H(x, y)=0\quad\mbox{in}\,\, \Omega\qquad H(x, y)=\frac{1}{|x-y|^4}\,\, x\in\partial\Omega.$$
It is known that the following expansion holds (see \cite{rey})
\begin{equation}\label{varphiexp}PU_{\delta,\xi}(x)=U_{\delta,\xi}(x)-\alpha_6 \delta^2 H(x, \xi)+\mathcal O\(\delta^4\)\ \hbox{as}\ \delta\to0\end{equation}
uniformly with respect to $\xi$ in compact sets of $\Omega$

Moreover we recall (see \cite{B}) that every solution to the linear equation $$-\Delta\psi=2U_{\delta, \xi}\psi\quad\mbox{in}\,\, \mathbb R$$ is a linear combination of the functions $Z_{\delta, \xi}^j$ $j=0, \ldots, 6$ given by $$Z_{\delta,\xi}^0(x)=\partial_\delta U_{\delta, \xi}(x)=2\alpha_6\delta \frac{|x-\xi|^2-\delta^2}{\(\delta^2+|x-\xi|^2\)^3}$$ and $$Z_{\delta, \xi}^j(x)=\partial_{\xi_j}U_{\delta, \xi}(x)=4\alpha_6\delta^2\frac{x_j-\xi_j}{\(\delta^2+|x-\xi|^2\)^3}\qquad j=1, \ldots, 6.$$ If we denote by $P Z_{\delta, \xi}^j$  the projection of $Z_{\delta, \xi}^j$ onto $H^1_0(\Omega)$, i.e.
$$-\Delta PZ_{\delta,\xi}^j= f'(U_{\delta, \xi})Z_{\delta,\xi}^j\quad\mbox{in}\,\, \Omega,\,\,   PZ_{\delta,\xi}^j=0\quad\mbox{on}\,\,\, \partial\Omega, $$ elliptic estimates give $$PZ_{\delta, \xi}^0(x)=Z_{\delta, \xi}^0-2\delta \alpha_6 H(x, \xi)+\mathcal O\(\delta^3\)\ \hbox{as}\ \delta\to0$$
and 
 $$PZ_{\delta,\xi}^j(x)=Z_{\delta,\xi}^j-\delta^2\alpha_6\partial_{\xi_j}H(x, \xi)+\mathcal O\(\delta^4\), \quad j=1, \ldots, 6 \ \hbox{as}\ \delta\to0$$  
 uniformly with respect to $\xi$ in compact sets of $\Omega.$ \\


We  look for a solution of \eqref{p1} of the form
\begin{equation}\label{sol}
u_\epsilon(x)=\underbrace{u_0(x)+\epsilon v_0-PU_{\delta, \xi}(x)}_{:=W_{\delta, \xi}}+\phi_\epsilon(x)
\end{equation}
where $\delta, \xi$ are chosen so that
\begin{equation}\label{sceltapar}
 \delta=|\epsilon| d\ \hbox{with}\ d\in\(\sigma, \frac{1}{\sigma}\)\ \hbox{and}\ \xi=\xi_0+\sqrt\delta \eta\ \hbox{with}\  |\eta|\le \frac{1}{\sigma}\ \hbox{where $\sigma>0$ is small}
\end{equation} and $\phi_\epsilon$ is a remainder term which is small as $\epsilon\to0$ which belongs to the space $\mathcal K_{\delta, \xi}^\bot$ defined as follows.

Now let us define $$\mathcal K_{\delta, \xi}:={\rm span}\{P Z_{\delta,\xi}^j\,\,:\,\, j=0, \ldots, 6\}$$ and $$\mathcal K_{\delta, \xi}^\bot:=\{\phi\in H^1_0(\Omega)\,\,:\,\,\, (\phi, PZ_{\delta,\xi}^j)=0\,\, j=0, \ldots, 6\}.$$
Let us denote by $\Pi_{\delta, \xi}$ and $\Pi_{\delta,\xi}^\bot$ the projection of $H^1_0(\Omega)$ on $\mathcal K_{\delta,\xi}$ and $\mathcal K_{\delta,\xi}^\bot$ respectively.\\
Then solves problem \eqref{pb1r} is equivalent to solve the system
\begin{equation}\label{sist1}
\Pi_{\delta,\xi}^\bot\left\{u_\epsilon(x)-(-\Delta)^{-1}\left[f(u_\epsilon)+\lambda u_\epsilon\right]\right\}=0\end{equation}
\begin{equation}\label{sist2}
\Pi_{\delta,\xi}\left\{u_\epsilon(x)-(-\Delta)^{-1}\left[f(u_\epsilon)+\lambda u_\epsilon\right]\right\}=0\end{equation}\vskip0.2cm
\subsection{The remainder term: solving  equation (\ref{sist1})}
The equation \eqref{sist1} can be written as 
$$
\mathcal L_{\delta,\xi}(\phi_\epsilon)+\mathcal R_{\delta,\xi}+\mathcal N_{\delta,\xi}(\phi_\epsilon)=0
$$
where
$$
\mathcal L_{\delta,\xi}(\phi_\epsilon)=\Pi_{\delta,\xi}^\bot\left\{\phi_\epsilon(x)-(-\Delta)^{-1}\left[f'(W_{\delta, \xi})\phi_\epsilon+\lambda \phi_\epsilon\right]\right\}
$$
 is the linearized operator at the approximate solution,
 $$
\mathcal R_{\delta,\xi}=\Pi_{\delta,\xi}^\bot\left\{W_{\delta, \xi}(x)-(-\Delta)^{-1}\left[f(W_{\delta, \xi})+\lambda W_{\delta, \xi}\right]\right\}$$
is the error term and  $$
\mathcal N_{\delta,\xi}(\phi_\epsilon)=\Pi_{\delta,\xi}^\bot\left\{-(-\Delta)^{-1}\left[f(W_{\delta, \xi}+\phi_\epsilon)-f(W_{\delta, \xi})-f'(W_{\delta, \xi})\phi_\epsilon\right]\right\}$$
 is a quadratic term in $\phi_\epsilon$.\\
 
First of all, we estimate the size of the error term $\mathcal R_{\delta,\xi}.$ 
\begin{lemma}\label{error}For any $\sigma>0$ 
there exist  $c>0$ and $\varepsilon_0>0$ such that for any $d>0$ and $\eta\in\mathbb R$ satisfying \eqref{sceltapar} and for any $\epsilon\in(-\varepsilon_0,\varepsilon_0)$
$$\|\mathcal R_{\delta,\xi} \|\le c \e^{2}|\ln|\e||^{\frac 23}.$$\end{lemma}
\begin{proof}
First we remark that
$$\begin{aligned}
&-\Delta W_{\delta,\xi}-|W_{\delta,\xi}|W_{\delta,\xi}-(\lambda_0+\e)W_{\delta,\xi}\\ &=-\Delta u_0-\e\Delta v_0-U_{\de, \xi}^2-|u_0+\e v_0- PU_{\de,\xi}|(u_0+\e v_0- PU_{\de,\xi})\\
&-\lambda_0 u_0-\lambda_0\e v_0+(\lambda_0+\e)PU_{\de,\xi}-\e u_0-\e^2v_0\\
&=-|u_0+\e v_0- PU_{\de,\xi}|(u_0+\e v_0- PU_{\de,\xi})-U_{\de, \xi}^2+|u_0|u_0\\ &+\e\underbrace{\(-\Delta v_0 -\lambda_0 v_0 - u_0\)}_{=2|u_0| v_0\ \hbox{\tiny because of \eqref{v0}}}+(\lambda_0+\e)PU_{\de,\xi}-\e^2v_0. \end{aligned}$$
By the continuity of $\Pi_{\delta,\xi}^\bot$ we get that
$$\begin{aligned} \|\mathcal R_{\delta,\xi}\|&\le c\left|-\Delta W_{\delta,\xi}-f(W_{\delta,\xi})-\lambda W_{\delta,\xi}\right|_{\frac 32}\\
&\le c \underbrace{\left|-|u_0+\e v_0- PU_{\de,\xi}|(u_0+\e v_0- PU_{\de,\xi})-PU_{\delta,\xi}^2+|u_0|u_0+2\e|u_0|v_0\right|_{\frac 32}}_{(I)}\\ &+c\underbrace{\left|PU_{\delta,\xi}^2-U_{\delta,\xi}^2\right|_{\frac 32}}_{(II)}\\
&+(\lambda_0+\e)\left|PU_{\delta,\xi}\right|_{\frac 32}+\underbrace{\e^2\left|v_0\right|_{\frac 32}}_{:=\mathcal O\(\e^2\)}
\end{aligned}$$
First of all, we point out that
$$|PU_{\delta,\xi}|_{\frac 32}\le c |U_{\delta,\xi}|_{\frac 32}\le c \de^2|\ln\de|^{\frac 23}.$$
and by \eqref{varphiexp}
 $$\begin{aligned}
(II)&\le c\(\int_\Omega\underbrace{|PU_{\delta,\xi}-U_{\delta,\xi}|^{\frac 32}}_{=O(\delta^2)}\underbrace{|PU_{\delta,\xi}+U_{\delta,\xi}|^{\frac 32}}_{\le cU_{\delta,\xi}}\)^{\frac 23}\le c \delta^2 \(\int_\Omega |U_{\delta,\xi}|^{\frac 32}\, dx\)^{\frac 23} =\mathcal O\( \de^4|\ln\de|^{\frac 23}\).\end{aligned}$$
First let us estimate $(I)$ in $B(\xi, \sqrt\de)$ and $\Omega\setminus B(\xi, \sqrt\de)$:
$$\begin{aligned} (I)&\le c \(\int_{B(\xi, \sqrt\de)}\big||u_0+\e v_0-PU_{\de,\xi}|(u_0+\e v_0-PU_{\de,\xi})|+(PU_{\de,\xi})^2\big|^{\frac 32}\)^{\frac 23}\\ &+c\underbrace{\(\int_{B(\xi, \sqrt\de)}\big||u_0|u_0+2\e |u_0| v_0|^{\frac 32}\, dx\)^{\frac 23}}_{=\mathcal O(\delta^2)}\\
&+ c\(\int_{\Omega\setminus B(\xi, \sqrt \de)}\big||u_0+\e v_0-PU_{\de, \xi}|(u_0+\e v_0-PU_{\de, \xi})-|u_0|u_0-  2|u_0| (\e v_0-PU_{\de, \xi})\big|^{\frac 32}\)^{\frac 23}\\ 
&+c\(\int_{\Omega\setminus B(\xi, \sqrt \de)}\big|(PU_{\de,\xi})^2+2|u_0|PU_{\de, \xi}\big|^{\frac 32}\, dx\)^{\frac 23}\\
&=\mathcal O\(\delta^2|\ln\de|^{\frac 23}\),\end{aligned}$$
since
 by mean value Theorem (here $\theta\in[0,1]$)  
$$\begin{aligned}&\int_{B(\xi, \sqrt\de)}\big||u_0+\e v_0-PU_{\de,\xi}|(u_0+\e v_0-PU_{\de,\xi})+(PU_{\de,\xi})^2\big|^{\frac 32}
\\
 &=2\int_{B(\xi, \sqrt\de)}|(\theta(u_0+\e v_0)-PU_{\de,\xi})(u_0+\e v_0)|^{\frac 32}\, dx
\\ &
\le c\underbrace{\int_{B(\xi, \sqrt\de)}|PU_{\de,\xi}|^{\frac32}\, dx}_{=\mathcal O(\delta^3|\log\delta|)}+c\underbrace{\int_{B(\xi, \sqrt\de)}|u_0+\e v_0|^{3}\, dx}_{=\mathcal O(\delta^3)}
\end{aligned},$$
and by the inequality
 \begin{equation}\label{a1}
 \big||a+b|(a+b)-|a|a-2|a|b\big|\le 7 b^2\ \hbox{for any}\ a,b\in \mathbb R
 \end{equation}
 $$\begin{aligned}&\int_{\Omega\setminus B(\xi, \sqrt \de)}\Big||u_0+\e v_0-PU_{\de, \xi}|(u_0+\e v_0-PU_{\de, \xi})-|u_0|u_0-  2|u_0| (\e v_0-PU_{\de, \xi})\Big|^{\frac 32}\\
 &\le c\int_{\Omega\setminus B(\xi, \sqrt \de)}|\e v_0-PU_{\de, \xi}|^{3}dx\\
 &\le c\underbrace{\int_{\Omega\setminus B(\xi, \sqrt \de)}|\e v_0|^{3}dx}_{=\mathcal O(\epsilon^3)}+c\underbrace{\int_{\Omega\setminus B(\xi, \sqrt \de)}| U_{\de, \xi}|^{3}dx}_{=\mathcal O(\delta^3)},\\
 &\left(\int_{\Omega\setminus B(\xi, \sqrt \de)}\Big||u_0+\e v_0-PU_{\de, \xi}|(u_0+\e v_0-PU_{\de, \xi})-|u_0|u_0-  2|u_0| (\e v_0-PU_{\de, \xi})\Big|^{\frac 32}\right)
 ^\frac23=O(\e^2)
\end{aligned}$$
 and 
 $$\int_{\Omega\setminus B(\xi, \sqrt \de)}\Big||PU_{\de, \xi}|(PU_{\de, \xi})+2|u_0|PU_{\de, \xi}\Big|^{\frac 32}\le c\underbrace{\int_{\Omega\setminus B(\xi, \sqrt \de)}| U_{\de, \xi} |^3\, dx}_{=\mathcal O(\delta^3 )}+\underbrace{\int_{\Omega\setminus B(\xi, \sqrt \de)}| U_{\de, \xi}|^{\frac 32}\, dx}_{=\mathcal O(\delta^3|\log\delta|)}.$$
which ends the proof.
 \end{proof}
Next we analyze the invertibility of the linear operator $\mathcal L_{\delta,\xi}$  (see  for example \cite{va}, Lemma 2.4 or \cite{rv}, Lemma 4.2).
\begin{lemma}\label{inv}
For any $\sigma>0$ there exist $c>0$ and $\varepsilon_0>0$ such that for any $d>0$ and $\eta\in\mathbb R$ satisfying \eqref{sceltapar} and for any $\e\in(-\varepsilon_0,\varepsilon_0)$ 
$$\|\mathcal L_{\delta,\xi}(\phi) \|\ge c \|\phi\|\ \hbox{for any}\ \phi\in \mathcal K_{\de,\xi}^\bot.$$ Moreover, $\mathcal L_{\de,\xi}$ is invertible and $\|\mathcal L_{\de,\xi}^{-1}\|\le \frac 1 c.$\end{lemma}
We are in position now to find a solution of the equation \eqref{sist1} whose proof relies on a standard contraction mapping argument (see for example \cite{mupi}, Proposition 1.8 and \cite{mipive}, Proposition 2.1)
\begin{proposition}\label{solphi}
For any $\sigma>0$ there exist $c>0$ and $\varepsilon_0>0$ such that for any $d>0$ and $\eta\in\mathbb R$ satisfying \eqref{sceltapar} and for any $\e\in(-\varepsilon_0,\varepsilon_0)$ 
 there exists a unique $\phi_\epsilon=\phi_\epsilon(d,\eta)\in \mathcal K_{\de,\xi}^\bot$ solution to \eqref{sist1} which is continuously differentiable with respect to $d$ and $\eta$ and such that
$$
\|\phi_\epsilon\|\le c \e^2|\ln|\e||^{\frac 23}. 
$$ \end{proposition}

\subsection{The reduced problem: solving equation (\ref{sist2})}
To solve  equation \eqref{sist2}, we shall  find the parameter $\delta$  and the point $\xi\in\Omega$ as in \eqref{sceltapar},  i.e. $d>0$ and  $\eta\in\mathbb R,$  so that \eqref{sist2} is satisfied. \\ It is well known that this problem has a variational structure, in the sense that solutions of \eqref{sist2} reduces to find critical points to some given explicit  finite dimensional functional. Indeed, let $J_\e: H^1_0(\Omega)\to \mathbb R$ defined by $$J_\e(u):=\frac 12 \int_\Omega |\nabla u|^2\, dx-\frac \lambda 2 \int_\Omega u^2\, dx-\frac 13\int_\Omega |u|^3\, dx$$ and let $\tilde J_\e: \mathbb R_+\times \mathbb R\to \mathbb R$ be the reduced energy which is defined by $$\tilde J_\e(d, \eta)=J_\e(W_{\delta, \xi}+\phi_\epsilon).$$

\begin{proposition}\label{cruc0} For any $\sigma>0$ there exists $\varepsilon_0>0$ such that for any $\epsilon\in(-\varepsilon_0,\varepsilon_0)$ 
\begin{equation}\label{cruc1}
\tilde J_\e(d,\eta)= \mathfrak c_0(\e)+|\e|^3 \Upsilon(d,\eta) +o\(|\e|^3\)
\end{equation}
with
\begin{equation}\label{cruc2}
\Upsilon(d,\eta):=  \mathtt{sgn}(\e)\(1-2v_0(\xi_0)\) d^2\mathfrak a_1+d^3\(\mathfrak a_2 \langle D^2 u_0(\xi_0)\eta, \eta\rangle-\mathfrak a_3 \),
\end{equation}
uniformly with respect to $(d,\eta)$ which satisfies \eqref{sceltapar}, where the
  $\mathfrak c_0(\e)$ only depends on $\e$ and the $\mathfrak a_i$'s are positive constants.
Moreover,  if $(d, \eta)$  is a critical point of $\tilde J_\e$, then $W_{\delta,\xi}+\phi_\epsilon$ is a solution of \eqref{p1}.\\
\end{proposition}

\begin{proof}
It is quite standard to prove that  if $(d, \eta)$ satisfies \eqref{sceltapar} and is a critical point of $\tilde J_\e$, then $W_{\delta,\xi}+\phi_\epsilon$ is a solution of \eqref{p1}
(see for example \cite{mipive}, Proposition 2.2). Moreover, it is not difficult to check that $\tilde J_\e(d,\eta)=J_\e(W_{\delta,\xi})+o\(|\e|^3\) $ uniformly with respect to $(d, \eta)$  which
 satisfies \eqref{sceltapar} (see for example \cite{mipive}, Proposition 2.2). \\
We need only to estimate  the main term of the reduced energy $J_\e(W_{\delta,\xi})$, i.e. 
\begin{align*} &J_\e(u_0+\e v_0-PU_{\de,\xi})\\ &= \frac 1 2 \intO |\nabla(u_0+\e v_0- PU_{\de,\xi})|^2-\frac{\lambda_0+\e}{2}\intO(u_0+\e v_0- PU_{\de,\xi})^2
-\frac 13\intO |u_0+\e v_0- PU_{\de,\xi}|^3	\\
&=\frac 1 2 \intO |\nabla(u_0+\e v_0)|^2+\frac 12 \intO|\nabla PU_{\de,\xi}|^2-\frac{\lambda_0+\e}{2}\intO(u_0+\e v_0)^2-\frac{\lambda_0+\e}{2}\intO(PU_{\de,\xi})^2\\  &-\underbrace{\(\intO \nabla u_0\nabla PU_{\de,\xi}-\lambda_0\intO u_0 PU_{\de,\xi}\)}_{=\intO |u_0|u_0 PU_{\de,\xi}}-\e\underbrace{ \(\intO \nabla v_0\nabla PU_{\de,\xi}-\lambda_0  \intO v_0 PU_{\de,\xi}- \intO u_0 PU_{\de,\xi}\)}_{= \intO 2|u_0| v_0  PU_{\de,\xi} }   \\
&+ \e^2\intO  v_0PU_{\de,\xi}\\
&-\frac 13\intO |u_0+\e v_0- PU_{\de,\xi}|^3\\
&  =\underbrace{\frac 1 2 \intO |\nabla(u_0+\e v_0)|^2-\frac{\lambda_0+\e}{2}\intO(u_0+\e v_0)^2-\frac13 \intO |u_0+\e v_0 |^3}_{=:I_1} \\
&+ \underbrace{\frac 12 \intO|\nabla PU_{\de,\xi}|^2-\frac 13\intO PU_{\de,\xi}^3}_{=:I_2}\underbrace{-\frac{\lambda_0}{2}\intO  PU_{\de,\xi}^2+\intO u_0  PU_{\de,\xi}^2}_{=:I_3}
\underbrace{-\frac\e 2\intO PU_{\de,\xi}^2+\e\int v_0 PU_{\de,\xi}^2}_{=:I_4}\\ 
&\underbrace{ -\frac 13\intO\(|u_0+\e v_0- PU_{\de,\xi}|^3-|u_0+\e v_0|^3- PU_{\de,\xi}^3+3(u_0  +\e v_0)PU_{\de,\xi}^2+3|u_0 +\e u_0|(u_0+\e v_0) PU_{\de,\xi}\)}_{=:I_5}\\&
\underbrace{ + \intO\Big[|u_0 +\e v_0|(u_0+\e v_0) -(|u_0|u_0 +2\e|u_0|  v_0)\Big] PU_{\de,\xi} }_{=:I_6}+ \underbrace{\e^2\intO v_0PU_{\de,\xi}}_{=:I_7}\end{align*}
It is clear that
$$I_7=\mathcal O\(\e^2\intO{\delta^2\over |x-\xi|^4}dx\right)=\mathcal O\(\e^2\delta^2\)=\mathcal O\(\e^4\).$$
To estimate $I_6$ by \eqref{a1} it follows that
$$I_6=O\left(\e^2\intO PU_{\de,\xi}
\right)=O\(\e^2\delta^2\)=\mathcal O\(\e^4\).$$
Now, $I_1$  does not depend neither on $d$ nor on $\eta$  and it will be included in the constant $\mathfrak c_0$ in \eqref{cruc1}.
By \eqref{varphiexp}
\begin{equation*}
\begin{aligned}
I_2&=\frac 12 \intO U_{\de,\xi}^3-\frac 13\intO PU_{\de,\xi}^3\\ &=\frac 12 \intO U_{\de,\xi}^3-\frac 13\intO\big(U_{\delta,\xi}(x)-\alpha_6 \delta^2 H(x, \xi)+\mathcal O\(\delta^4\)\big)^3\\
&=
\frac1 6
\int\limits_{\mathbb R}U^3+\mathcal O\left(\delta^2\intO U_{\delta,\xi}^2\right)+O\(\delta^4\)\\ &=\frac1 6
\int\limits_{\mathbb R}U_{\delta,\xi}^3+O\(\delta^4\).
 \end{aligned}
\end{equation*}\\
Now, 
 setting $\varphi_{\delta,\xi}:=PU_{\delta, \xi}-U_{\delta, \xi}=\mathcal O(\delta^2),$ by \eqref{varphiexp} and \eqref{sceltapar}
\begin{equation*}
\begin{aligned}
I_3&=\intO \( u_0(x)-\frac{\lambda_0}2\)(U_{\delta, \xi}+\varphi_{\delta,\xi})^2\\
& =\intO \( u_0(x)-u_0(\xi_0)\)U_{\delta, \xi}^2+\mathcal O(\delta^4)\\
&=\intO\left[\frac12\langle D^2 u_0(\xi_0)(x-\xi_0),(x-\xi_0)\rangle+\mathcal O(|x-\xi_0|^3)\right]\alpha_6^2\frac{\delta^4}{(\delta^2+|x-\xi|^2)^4}dx +\mathcal O(\delta^4)\\
&=\alpha_6^2\intO\frac12 \langle D^2 u_0(\xi_0)(x-\xi_0),(x-\xi_0)\rangle \frac{\delta^4}{(\delta^2+|x-\xi|^2)^4}dx +\mathcal O(\delta^4)\\
&=\alpha_6^2\delta^2\int\limits_{\Omega-\xi\over\delta}\frac12 \langle D^2 u_0(\xi_0)(\delta y+\sqrt\delta\eta),(\delta y+\sqrt\delta\eta)\rangle \frac{1}{(1+|y|^2)^4}dy +\mathcal O(\delta^4)\\
&=\frac{\alpha_6^2}2\delta^3\(\int\limits_{\mathbb R}\frac{1}{(1+|y|^2)^4}dy\)  \langle D^2 u_0(\xi_0)\eta,\eta\rangle +\mathcal O(\delta^4|\ln\delta|)\\
&=\frac{\alpha_6^2}2d^3|\e|^3\(\int\limits_{\mathbb R}\frac{1}{(1+|y|^2)^4}dy\)  \langle D^2 u_0(\xi_0)\eta,\eta\rangle +\mathcal O(\e^4|\ln|\e||).
 \end{aligned}
\end{equation*}
and analogously 
\begin{equation*}
\begin{aligned}
 I_4&=\e\intO\(v_0(x)-\frac12\) PU_{\de,\xi}^2\\ &=\e\left[\alpha_6^2\delta^2\(\int\limits_{\mathbb R}\frac{1}{(1+|y|^2)^4}dy\) \(v_0(\xi_0)-\frac12\)  +o(1)\right]\\
&=\e^3d^2\left[\alpha_6^2\(\int\limits_{\mathbb R}\frac{1}{(1+|y|^2)^4}dy\) \(v_0(\xi_0)-\frac12\)  +o(1)\right]
 \end{aligned}
\end{equation*}

Finally, we have to estimate $I_5.$

We point out that
$$|u_0+\e v_0- PU_{\de,\xi}|^3-|u_0+\e v_0|^3- PU_{\de,\xi}^3+3(u_0  +\e v_0)PU_{\de,\xi}^2+3|u_0 +\e u_0|(u_0+\e v_0) PU_{\de,\xi}=0\ \hbox{if}\ u_0+\e v_0\le0$$
and so
$$\begin{aligned} I_5  &=-\frac 13\int_{\{u_0+\e v_0\ge0\}}\(|u_0+\e v_0- PU_{\de,\xi}|^3-(u_0+\e v_0)^3- PU_{\de,\xi}^3\right.\\ &\left.
\hskip 4truecm+3(u_0  +\e v_0)PU_{\de,\xi}^2+
3 (u_0+\e v_0)^2 PU_{\de,\xi}\)dx \\
&=-\frac 13\int_{\{u_0+\e v_0\ge PU_{\delta, \xi}\}}\left(-2PU_{\delta, \xi}^3+6(u_0  +\e v_0) PU_{\de,\xi}^2\right)\\
&-\frac 13\int_{\{0< u_0+\e v_0<PU_{\delta, \xi}\}}\left(-2(u_0+\e v_0)^3+6(u_0  +\e v_0)^2 PU_{\de,\xi} \right).\end{aligned}$$
%
 
First of all we claim that for any $\sigma>0$ there exists $\varepsilon_0>0$ such that for any $\e\in(-\varepsilon_0,\varepsilon_0)$ and $(d,\xi)$ satisfying \eqref{sceltapar}
\begin{equation}\label{claimlevel}
B\(\xi, R^1_\delta\sqrt\delta\)\subset\{x\in\Omega\,:\, 0< u_0(x)+\e v_0(x)< P U_{\delta, \xi}(x)\} \cap B\(\xi,\delta^\frac14\)\subset B\(\xi, R^2_\delta\sqrt\delta\)
\end{equation}
where
\begin{equation}\label{ro}  R^1_\delta,R^2_\delta=R_0+o(1)\ \hbox{with}\  R_0:=\(\frac{\alpha_6}{u_0(\xi_0)}\)^{\frac 14}. \end{equation}
We remind that $\delta=\mathcal O(\epsilon)$ and also that   $P U_{\delta, \xi}(x)=\alpha_6{\delta^2\over(\delta^2+|x-\xi|^2)^2}+\mathcal O(\epsilon^2)$ uniformly in $\Omega.$
If $|x-\xi|<R^1_\delta\sqrt\delta$ is small enough then by mean value theorem $u_0(x)+\e v_0(x)=u_0(\xi_0)+\mathcal O_1(\epsilon) $ and
$$ \begin{aligned}u_0(x)+\e v_0(x)< P U_{\delta, \xi}(x)\ &\Leftrightarrow\ \frac{u_0(\xi_0)}{\alpha_6}+\mathcal O_1(\epsilon)< {\delta^2\over(\delta^2+|x-\xi|^2)^2}\\
& \Leftrightarrow\ 
|x-\xi|\le \sqrt\delta \underbrace{\({1\over  \(\frac{u_0(\xi_0)}{\alpha_6}+\mathcal O_1(\epsilon)\)^{\frac12}}-\delta\)^{1\over2}}_{R^1_\delta}\end{aligned}$$
and the first inclusion in \eqref{claimlevel} together with \eqref{ro} follow.
On the other hand, again by mean value theorem 
we have  $u_0(x)+\e v_0(x)=u_0(\xi_0)+\mathcal O_2(\sqrt \delta)$ for any $x\in B(\xi,\delta^\frac14)$  
and arguing as above we get the second inclusion in \eqref{claimlevel} and \eqref{ro}. \\

It is useful to point out that by \eqref{claimlevel} we immediately get
\begin{equation}\label{claimlevel+}
B^c\(\xi, R^1_\delta\sqrt\delta\) \supset 
\{x\in\Omega\,:\,   u_0(x)+\e v_0(x)\ge P U_{\delta, \xi}(x)\}\cup B^c\(\xi,\delta^\frac14\)
\supset B^c\(\xi, R^2_\delta\sqrt\delta\)
\end{equation}
 Now by \eqref{claimlevel} and \eqref{claimlevel+} we deduce
$$\begin{aligned} I_5&=-\frac 13\int_{\{u_0+\e v_0\ge PU_{\delta, \xi}\}}\left(-2PU_{\delta, \xi}^3+6(u_0  +\e v_0) PU_{\de,\xi}^2\right)\\
&-\frac 13\int_{\{0<u_0+\e v_0<PU_{\delta, \xi}\}}\left(-2(u_0+\e v_0)^3+6(u_0  +\e v_0)^2 PU_{\de,\xi} \right)\\
&=-\frac 13\int_{\{u_0+\e v_0\ge PU_{\delta, \xi}\}\cup B^c\(\xi,\delta^\frac14\)}\left(-2PU_{\delta, \xi}^3+6(u_0  +\e v_0) PU_{\de,\xi}^2\right)\\
&+\frac 13\int_{B^c\(\xi,\delta^\frac14\)\setminus \{u_0+\e v_0\ge PU_{\delta, \xi}\}\cap B^c\(\xi,\delta^\frac14\) }\left(-2PU_{\delta, \xi}^3+6(u_0  +\e v_0) PU_{\de,\xi}^2\right)\\
&-\frac 13\int_{\{0<u_0+\e v_0<PU_{\delta, \xi}\}\cap B\(\xi,\delta^\frac14\)}\left(-2(u_0+\e v_0)^3+6(u_0  +\e v_0)^2 PU_{\de,\xi} \right)\\
& -\frac 13\int_{\{0<u_0+\e v_0<PU_{\delta, \xi}\}\cap B^c\(\xi,\delta^\frac14\)}\left(-2(u_0+\e v_0)^3+6(u_0  +\e v_0)^2 PU_{\de,\xi} \right)\\
  &=-\frac 13\int_{\{u_0+\e v_0\ge PU_{\delta, \xi}\}\cup B^c\(\xi,\delta^\frac14\)  }\left(-2PU_{\delta, \xi}^3+6(u_0  +\e v_0) PU_{\de,\xi}^2\right)\\ 
  &-\frac 13\int_{\{0<u_0+\e v_0<PU_{\delta, \xi}\} \cap B\(\xi,\delta^\frac14\)}\left(-2(u_0+\e v_0)^3+6(u_0  +\e v_0)^2 PU_{\de,\xi} \right)+o(\delta^3),
\end{aligned}$$
because
$$\begin{aligned}& \int_{B^c\(\xi,\delta^\frac14\)\setminus\{u_0+\e v_0\ge PU_{\delta, \xi}\}\cap B^c\(\xi,\delta^\frac14\)}\left(-2PU_{\delta, \xi}^3+6(u_0+\e v_0) PU_{\delta, \xi}^2  \right)\\ &=\mathcal O\(
\int_{  B^c\(\xi,\delta^\frac14\)}\left(U_{\delta, \xi}^3+ U_{\delta, \xi}^2  \right)\)\\
&=\mathcal O\(\delta^\frac72\),\end{aligned}$$
and
$$\begin{aligned}&\int_{\{0<u_0+\e v_0<PU_{\delta, \xi}\}\cap B^c\(\xi,\delta^\frac14\)}\left(-2(u_0+\e v_0)^3+6(u_0  +\e v_0)^2 PU_{\de,\xi} \right)\\ &=\mathcal O\(\delta^3\mathtt{meas}\{0<u_0(x)<2\delta\}\) \\
&=o(\delta^3)
\end{aligned}
$$
since
$PU_{\delta, \xi}(x)=\mathcal O(\delta)$ if $|x-\xi|\ge \delta^\frac14$ and  $\{0<u_0+\e v_0<PU_{\delta, \xi}\}\cap B^c\(\xi,\delta^\frac14\)\subset \{0<u_0(x)<2\delta\}$ if $\delta$ is small enough. 
Next we claim that
$$
\begin{aligned}
&-\frac 13\int_{\{u_0+\e v_0\ge PU_{\delta, \xi}\}\cup B^c\(\xi,\delta^\frac14\)  }\left(-2PU_{\delta, \xi}^3+6(u_0  +\e v_0) PU_{\de,\xi}^2\right)\\ &-\frac 13\int_{\{0<u_0+\e v_0<PU_{\delta, \xi}\} \cap B\(\xi,\delta^\frac14\)}\left(-2(u_0+\e v_0)^3+6(u_0  +\e v_0)^2 PU_{\de,\xi} \right)+o(\delta^3) \\
 &=-\frac 13\int_{\{u_0+\e v_0\ge PU_{\delta, \xi}\}\cup B^c\(\xi,\delta^\frac14\)  }\left(-2PU_{\delta, \xi}^3+6u_0 PU_{\delta, \xi}^2  \right)\\ &-\frac 13\int_{\{0<u_0+\e v_0<PU_{\delta, \xi}\} \cap B\(\xi,\delta^\frac14\)}\left(-2u_0^3+6u_0^2 PU_{\delta, \xi}  \right)+o(\delta^3),\\
\end{aligned}
$$
Indeed using \eqref{claimlevel+} and \eqref{claimlevel} we get
$$\int_{\{u_0+\e v_0\ge PU_{\delta, \xi}\}\cup B^c\(\xi,\delta^\frac14\)  }   PU_{\de,\xi}^2=\mathcal O\(
\int_{  B^c\(\xi,\delta^\frac12\)} U_{\delta, \xi}^2 \)=\mathcal O\(\delta^3\),$$
$\mathtt{meas} B \(\xi,\delta^\frac12\)=\mathcal O(\delta^3)$ and
$$\int_{\{u_0+\e v_0<PU_{\delta, \xi}\}\cap B \(\xi,\delta^\frac14\)  }  PU_{\de,\xi} =\mathcal O\(
\int_{  B \(\xi,\delta^\frac12\)} U_{\delta, \xi}  \)=\mathcal O\(\delta^3\).$$

We estimate the last two  terms in the expansion of $I_5.$ By \eqref{claimlevel+}
\begin{equation*}
B^c\(\xi, R^2_\delta\sqrt\delta\)\subset\{x\in\Omega\,:\, u_0(x)+\e v_0(x)\ge  P U_{\delta, \xi}\}\cup B^c\(\xi,\delta^\frac14\)\subset B^c\(\xi, R^1_\delta\sqrt\delta\)\
\end{equation*}
Hence
$$\begin{aligned}\int_{|x-\xi|>R_\delta^2\sqrt\de}\left(-2PU_{\delta, \xi}^3+6u_0 PU_{\delta, \xi}^2  \right)&\le \int_{\{u_0+\e v_0\ge PU_{\delta, \xi}\}\cup B\(\xi,\delta^\frac14\)}\left(-2PU_{\delta, \xi}^3+6u_0 PU_{\delta, \xi}^2  \right)\\&\le \int_{|x-\xi|>R_\de^1\sqrt\de}\left(-2PU_{\delta, \xi}^3+6u_0 PU_{\delta, \xi}^2  \right).\end{aligned}$$
Now if $R_\delta$ denotes either $R^1_\delta$ or $R^2_\delta$ we get
$$\begin{aligned}&\int_{|x-\xi|>R_\de\sqrt\de}\left(-2PU_{\delta, \xi}^3+6u_0 PU_{\delta, \xi}^2  \right)\\ &=-2\int_{|x-\xi|>R_\de\sqrt\de}U_{\de,\xi}^3+6\int_{|x-\xi|>R_\de\sqrt\de}u_0 U_{\de,\xi}^2+\mathcal O\(\delta^4\)\\
&  =-2\int_{|y|>\frac{R_\de}{\sqrt\de}}\frac{\alpha_6^3}{(1+|y|^2)^6}+6\de^2\int_{|y|>\frac{R_\de}{\sqrt\de}}u_0(\de y+\xi)\frac{\alpha_6^2}{(1+|y|^2)^4}+\mathcal O\(\de^4\)\\
& =-2\omega_6\alpha_6^3\int_{\frac{R_\de}{\sqrt\de}}^{+\infty}\frac{r^5}{(1+r^2)^6}+6\de^2\omega_6\alpha_6^2 u_0(\xi_0)\int_{\frac{R_\de}{\sqrt\de}}^{+\infty}\frac{r^5}{(1+r^2)^4}+\mathcal O\(\de^4\int_{\frac{R_\de}{\sqrt\de}}^{+\infty}\frac{r^7}{(1+r^2)^4}\)+\mathcal O\(\de^4\)\\
&= -\frac 13\omega_6\alpha_6^3 R_\de^{-6}\de^3+3\de^3\omega_6\alpha_6^2R_\de^{-2}u_0(\xi_0)+\mathcal O\(\de^4|\log\de|\)\\
&= -\frac 13\omega_6\alpha_6^3 R_0^{-6}\de^3+3\de^3\omega_6\alpha_6^2R_0^{-2}u_0(\xi_0)+o\(\de^{3}\), \ \hbox{because of \eqref{ro}}.
\end{aligned}$$
and by comparison
\begin{equation}\label{fuoripalla}
\int_{\{u_0+\e v_0\ge PU_{\delta, \xi}\}\cup B^c\(\xi,\delta^\frac14\)}\left(-2PU_{\delta, \xi}^3+6u_0 PU_{\delta, \xi}^2  \right)= -\frac 13\omega_6\alpha_6^3 (R_0)^{-6}\de^3+3\de^3\omega_6\alpha_6^2(R_0)^{-2}u_0(\xi_0)+o\(\de^{3}\).\end{equation}
In a similar way, by   \eqref{claimlevel} 
$$\begin{aligned}\int_{|x-\xi|<R_\delta^1\sqrt\de}\left(-2u_0^3+6u_0^2 PU_{\delta, \xi}  \right)&\le \int_{\{0< u_0+\e v_0<PU_{\delta, \xi}\}\cap  B\(\xi,\delta^\frac14\)}\left(-2u_0^3+6u_0^2 PU_{\delta, \xi} \right)\\ &\le \int_{|x-\xi|<R_\de^2\sqrt\de}\left(-2u_0^3+6u_0^2 PU_{\delta, \xi}  \right).\end{aligned}$$
and   if $R_\delta$ denotes either $R^1_\delta$ or $R^2_\delta$ we get
 $$\begin{aligned}&\int_{|x-\xi|<R_\de\sqrt\de}\(-2u_0^3+6u_0^2 PU_{\delta, \xi}\)\\ &=-2\de^6\int_{|y|<\frac{R_\de}{\sqrt\de}}u_0^3(\de y+\xi)+6\de^4\int_{|y|<\frac{R_\de}{\sqrt\de}}u_0^2(\de y+\xi)\frac{\alpha_6}{(1+|y|^2)^2}+\mathcal O\(\de^5\)\\
& =\(-2u_0^3(\xi_0)+\mathcal O\(\sqrt\de\)\)\de^6\omega_6\int_0^{\frac{R_\de}{\sqrt \de}}r^5+6\alpha_6\(u_0^2(\xi_0)+\mathcal O\(\sqrt\de\)\)\de^4\omega_6\int_0^{\frac{R_\de}{\sqrt\de}}\frac{r^5}{(1+r^2)^2}+\mathcal O\(\de^5\)\\
& =-2\de^3u_0^3(\xi_0)\omega_6R_\de^6+3\alpha_6\de^3u_0^2(\xi_0)\omega_6R_\de^2+\mathcal O\(\de^{\frac 72}\)\\
& =-2\de^3u_0^3(\xi_0)\omega_6R_0^6+3\alpha_6\de^3u_0^2(\xi_0)\omega_6R_0^2+o\(\de^3\), \ \hbox{because of \eqref{ro}}.
\end{aligned}$$
and by comparison
\begin{equation}\label{dentropalla}
\int_{\{u_0+\e v_0<PU_{\delta, \xi}\}\cap B\(\xi,\delta^\frac14\)}\left(-2u_0^3+6u_0^2 PU_{\delta, \xi} \right)=-2\de^3u_0^3(\xi_0)\omega_6R_0^6+3\alpha_6\de^3u_0^2(\xi_0)\omega_6R_0^2+o\(\de^{3}\)\end{equation}

Finally, by \eqref{dentropalla} and \eqref{fuoripalla}
 $$I_5=|\e|^3d^3\(-\frac{11}9\omega_6\alpha_6^{\frac 32}(u_0(\xi_0))^{\frac 32}+o(1)\)$$
Collecting all the previous estimates we get 
$$
\tilde J_\e(d,\eta)= \mathfrak c_0(\e)+|\e|^3\underbrace{\left\{ \mathtt{sgn}(\e)\(1-2v_0(\xi_0)\) d^2\mathfrak a_1+d^3\(\mathfrak a_2 \langle D^2 u_0(\xi_0)\eta, \eta\rangle-\mathfrak a_3 \)\right	\}}_{=:\Upsilon(d,\eta)}+o\(|\e|^3\)
$$
 with
$$\begin{aligned}
&\mathfrak a_1 =\alpha_6^2\(\int\limits_{\mathbb R^6}\frac{1}{(1+|y|^2)^4}dy\)=96\omega_6 \\
&\mathfrak a_2 =\frac{\alpha_6^2}2\int\limits_{\mathbb R}\frac{dy}{(1+|y|^2)^4}\\
&\mathfrak a_3 =\frac{11}{9}\omega_6\alpha_6^{\frac 32}(u_0(\xi_0))^{\frac 32}
\end{aligned}$$
and that concludes the proof.
\end{proof}

We are now in position to prove Theorem \ref{main1}.
\begin{proof}[\bf Proof of Theorem \ref{main1}: completed] The claim  follows by Proposition \ref{cruc0} taking into account that
 if $ \mathtt{sgn}(\e)\(1-2v_0(\xi_0)\)>0$ the function $\Upsilon$ has always an isolated maximum point
 $(d_0,0),$ with $ d_0:={2\mathfrak a_1\over 3\mathfrak a_3}\mathtt{sgn}(\e)\(1-2v_0(\xi_0)\)$, which is stable under uniform perturbations. 
\end{proof}

\section{A generic result}\label{gen}

Let $\Omega_0$ be a bounded and smooth domain in $\mathbb R^n$, we let $D$ be and open neighbourhood of $\overline{\Omega_0}$ and  $\alpha\in(0,1).$
There exists $\epsilon>0$ such that if $\theta\in C^{3,\alpha}(\overline D,\mathbb R^n)$ with $ \|\theta\|_{2,\alpha}\le\epsilon $ then $\Theta=I+\theta$ maps $\Omega_0$ in a one-to-one way onto the smooth domain $\Omega_\theta:=\Theta(\Omega_0)$ with boundary $\partial\Omega_\theta=\Theta(\partial\Omega_0).$\\
If $x\in\Omega_0$ we agree that $\hat x=\Theta x=(I+\theta)x\in\Omega_\theta,$ with $\theta\in V.$ If
$\hat u\in H^1_0(\Omega_\theta)\cap H^2(\Omega_\theta)$ then it is clear that   
 $u=\hat u\circ\Theta \in  H^1_0(\Omega_0)\cap H^2(\Omega_0).$
 \\
 Our result reads as follows.
\begin{theorem}\label{main}
The set
\begin{equation}\label{non-de}\begin{aligned}\Xi:=\big\{\theta\in C^{3,\alpha}(\overline D,\mathbb R^n)\ :\ &\hbox{if $\lambda>0$ and $u\in H^1_0(\Omega_\theta)$ solve }\\ 
&\Delta u+\lambda u+|u|^{4\over n-2}u=0\ \hbox{in}\ \Omega_\theta,\ u=0\ \hbox{on}\ \partial\Omega_\theta\\
& \hbox{then $u$ is non-degenerate} \big\}\end{aligned}\end{equation}
is a residual subset in $C^{3,\alpha}(\overline D,\mathbb R^n),$
 i.e. $C^{3,\alpha}(\overline D,\mathbb R^n)\setminus \Xi$ is a countable union of close subsets without interior points.
\end{theorem}

The proof relies on the  following abstract transversality theorem  (see \cite{Q,ST,U}).
 \begin{theorem}\label{tran}
 Let $X,Y,Z$ be three Banach spaces and $U\subset X,$ $V\subset Y$ open subsets.
 Let $F:U\times V\to Z$ be a $C^\alpha-$map with $\alpha\ge1.$ Assume that

 \begin{itemize}
 \item[i)] for any $y\in V$, $F(\cdot,y):U\to Z$ is a Fredholm map of index $l$ with $l\le\alpha;$
 \item[ii)] $0$ is a regular value of $F$, i.e. the operator $F'(x_0,y_0):X\times Y\to Z$ is onto at any point $(x_0,y_0)$ such that $F(x_0,y_0)=0;$
\item[iii)] the map  $\pi\circ i:F^{-1}(0)\to Y$ is $\sigma-$proper, i.e.  $F^{-1}(0)=\cup_{s=1}^{+\infty} C_s$
 where $C_s$ is a closed set and the restriction $\pi\circ i_{|_{C_s}}$ is proper for any $s$; here $i:F^{-1}(0)\to Y$ is the canonical embedding and $\pi:X\times Y\to Y$ is the projection.
 \end{itemize}

 Then the set
$\mathcal V:=\left\{y\in V\ :\ 0\ \hbox{is  a regular value of } F(\cdot,y)\right\}$
 is a  residual subset of $V$, i.e. $V\setminus \mathcal V$ is a countable union of close subsets without interior points.

\end{theorem}

Indeed, in our case we choose
$$\begin{aligned}
&X=\mathbb R\times \(H^1_0(\Omega_0)\cap H^2(\Omega_0)\)\\
&U=(0,\infty)\times \(H^1_0(\Omega_0)\cap H^2(\Omega_0)\setminus\{0\}\)\\
&Y= C^{3,\alpha}\(\overline D,\mathbb R^n\)\\
&V=\mathcal B_\epsilon:=\left\{\theta\in C^{3,\alpha}(\overline D,\mathbb R^n)\ :\ \|\theta\|_{2,\alpha}<\epsilon\right\}\\
&Z=\mathbb R\times L^2(\Omega_0).
\end{aligned}$$
$X$ and Z are  Banach spaces equipped with the norms
$\|(a,u)\|_X:=|a|+\|u\|_{H^1_0\cap H^2(\Omega_0)},$
and $\|(a,u)\|_Z:=|a|+\|u\|_{L^2(\Omega_0)},$
 respectively.
  Moreover,  the function $F:U\times V\to Z$ is defined by
$$
 F(\lambda,u,\theta):=\(Q (\lambda,\hat u,\theta),\Delta _{\hat x} \hat u+|\hat u|^{p-1}\hat u+\lambda \hat u\),
$$
where
$$
 Q (\lambda,\hat u,\theta):=\int\limits_{\Omega_\theta}\(|\nabla_{\hat x} \hat u|^2-|\hat u|^{p+1}-\lambda \hat u^2\)d\hat x.
$$
 It is clear that
 $$F(\lambda,u,\theta)=(0,0)\ \Leftrightarrow\ \Delta _{\hat x} \hat u+|\hat u|^{p-1}\hat u+\lambda \hat u=0\ \hbox{in}\ \Omega_\theta,\ \hat u=0\ \hbox{on}\ \partial\Omega_\theta. $$
 
 Theorem \ref{main} will follow by Theorem \ref{tran} as soon as we prove that $F$ satisfies the   assumptions and this is done below.\\
 
 First of all, we rewrite $F$ in terms of the $x-$variable (see \cite{ST,P})
 
\begin{lemma} We have
\begin{equation}\label{f1}
 Q (\lambda,\hat u,\theta):=\int\limits_{\Omega_0}\left\{\nabla u\cdot\left[(\det \Theta')(\Theta')^{-1}(^t\Theta')^{-1}\nabla u\right]-\(|u|^{p+1}+\lambda  u^2\)(\det \Theta')\right\}d x.
 \end{equation}
 and
\begin{equation}\label{f2}\Delta _{\hat x} \hat u+|\hat u|^{p-1}\hat u+\lambda \hat u=
{\mathrm {div}}\left[(\det \Theta')(\Theta')^{-1}(^t\Theta')^{-1}\nabla u\right]+\(| u|^{p-1}  u+\lambda  u\)(\det\Theta').
 \end{equation}
\end{lemma}

At this point it is useful to point out the following fact.
\begin{remark}\label{norm}
We can choose $\epsilon>0$ small enough so that for any $\theta\in \mathcal B_\epsilon$ 
$$\(\ \int\limits_{\Omega_0}\(\left| \left\langle(\det \Theta')(\Theta')^{-1}(^t\Theta')^{-1}\nabla u,\nabla u\right\rangle\right|^2
+\left|{\mathrm {div}}\left[(\det \Theta')(\Theta')^{-1}(^t\Theta')^{-1}\nabla u\right]\right|^2\)dx\)^{1/2}$$
defines on $H^1_0(\Omega_0)\cap H^2(\Omega_0)$ a norm which is equivalent to the standard one
$$\| u\|_{H^1_0\cap H^2(\Omega_0)}=\(\ \int\limits_{\Omega_0}\(|\nabla u|^2+|\Delta u|^2\)dx\)^{1/2}.$$
\end{remark}
 
 Next, we check the differentiability of $F$ (see \cite{ST,P}).
 
 \begin{lemma}
  The function $F$ is differentiable at any $(\lambda_0,u_0,\theta_0)\in U\times V$ such that $F(\lambda_0,u_0,\theta_0)=(0,0).$	
 Moreover if $\Theta_0=I+\theta_0$
\begin{equation}\label{f3}
\begin{aligned} &F'(\lambda_0,u_0,\theta_0)[\lambda,u]\\ &=\( \int\limits_{\Omega_0}\left\{2\nabla u_0\cdot\left[(\det \Theta_0')(\Theta_0')^{-1}(^t\Theta_0')^{-1}\nabla u\right]-\((p+1)|u_0|^{p-1}u_0+2\lambda _0 u_0\)u(\det \Theta_0')\right\}d x\right.\\ & \qquad -\lambda \int\limits_{\Omega_0}u_0^2(\det \Theta'_0)d x,\\
&\quad\left.\mathrm {div}\left[(\det \Theta_0')(\Theta_0')^{-1}(^t\Theta_0')^{-1}\nabla u\right]+\(p| u_0|^{p-1}  +\lambda_0  \)u(\det\Theta_0')+\lambda u_0(\det\Theta'_0),\)
\end{aligned}
 \end{equation}
 and if $\theta_0=0$
\begin{equation}\label{f4}
\begin{aligned} &F'(\lambda_0,u_0,\theta_0)[\theta]\\ &=\(\int\limits_{\Omega_0}\left\{\nabla u_0\cdot\left[(\mathrm {div}\theta)\nabla u_0-(\theta'+{}^t\theta')\nabla u_0\right]-\(|u_0|^{p+1}+ \lambda _0 u_0^2\)(\mathrm {div}\theta)\right\}d x\right.,\\
&\quad\left.{\mathrm {div}}\left[(\mathrm {div}\theta)\nabla u_0-(\theta'+{}^t\theta')\nabla u_0\right]+\(| u_0|^{p-1} u_0 +\lambda_0u_0  \)(\mathrm {div}\theta)\).
\end{aligned}
  \end{equation}
\end{lemma}
 
 Let us check assumption i) of Theorem \ref{tran}.
  \begin{lemma}
 For any $\theta\in V$ the function $F(\cdot,\cdot,\theta)$ is a Fredholm map from $U$ into $Z$ of index 0.
 \end{lemma}
 \begin{proof}
 The partial derivative $F'_{\lambda,u}(\lambda_0,u_0,\theta_0):X\to Z$ is the sum of   an isomorphism $\mathcal  I$  and a compact perturbation $\mathcal  K,$ namely
$$
  \mathcal I(\lambda,u):=
\(   -\lambda \int\limits_{\Omega_0}u_0^2(\det \Theta'_0)d x,\mathrm {div}\left[(\det \Theta_0')(\Theta_0')^{-1}(^t\Theta_0')^{-1}\nabla u\right]\)$$
and
$$\begin{aligned}
  \mathcal K(\lambda,u ):=&\( \int\limits_{\Omega_0}\left\{2\nabla u_0\cdot\left[(\det \Theta_0')(\Theta_0')^{-1}(^t\Theta_0')^{-1}\nabla u\right]-\((p+1)|u_0|^{p-1}u_0+2\lambda _0 u_0\)u(\det \Theta_0')\right\}d x\right.,\\
  &\left. \(p| u_0|^{p-1}  +\lambda_0  \)u(\det\Theta_0')+\lambda u_0(\det\Theta'_0)\).\end{aligned}
$$
 \end{proof}

 Let us check assumption iii) of Theorem \ref{tran}.
 \begin{lemma} The map $\pi\circ i: F^{-1}(0)\to Y$ is $\sigma-$proper.
 \end{lemma}
 \begin{proof}
Let us write
$$F^{-1}(0,0)=\cup_{m=1}^\infty \mathcal  C_m,\ \mathcal  C_m=\(A_m\times B_m\times C_m\)\cap F^{-1}(0,0),$$
 where 
 $$A_m:=\left\{\frac1m\le \lambda\le m\right\},$$
 $$\begin{aligned}B_m:&=\left\{u\in {H^1_0(\Omega_0)\cap H^2(\Omega_0)}\ :\ \frac1m\le \|u\|:=\(\ \int\limits_{\Omega_0}\(|\nabla u|^2+(\Delta u)^2\)dx\)^{\frac12}\le m
\right\}\end{aligned}$$
 and
 $$C_m:= \left\{\theta\in C^{3,\alpha}(\Omega_0)\ :\ \|\theta\|_{2,\alpha}\le \epsilon\(1-\frac1m\)\right\}. $$
 Let us fix $m$. We have to prove that if $(\theta_k)_{k\ge1}\subset C_m$ with $\theta_k\to\theta$ and $(\lambda_k,u_k)_{k\ge1}\subset A_m\times B_m$ is such that
 $F(\lambda_k,u_k,\theta_k)=0$ then, up to a subsequence, $(\lambda_k,u_k)\to(\lambda,u)\in A_m\times B_m$ and $F(\lambda,u,\theta)=0.$
 First of all, up to a subsequence, we have $\lambda_k\to\lambda\in A_m$ and $u_k\to u$ weakly in $H^1_0(\Omega_0)\cap H^2(\Omega_0)$ and strongly in $L^q(\Omega_0)$ for any $q>1$ if $n=3,4$ and $1<q<{2n\over n-4}$ if $n\ge5.$ 
 If $\Theta_k=I+\theta_k$ we know that $\Theta_k\to \Theta:=I+\theta$ in $ C^{1,\alpha}(\Omega_0,\mathbb R^n).$ Now, condition   $F(\lambda_k,u_k,\theta_k)=0$ reads as
 $${\mathrm {div}}\(\underbrace{(\det \Theta_k')(\Theta_k')^{-1}(^t\Theta_k')^{-1}}_{=A_k}\nabla u_k\)+\underbrace{\(| u_k|^{p-1}  u_k+\lambda _k u_k\) (\det\Theta_k')}_{=f_k}=0\ \hbox{in}\ \Omega_0,\ u=0\ \hbox{on}\ \partial\Omega_0.$$
In particular, for any $ \varphi\in H^1_0(\Omega_0)$
\begin{equation}\label{comp1}\int\limits_{\Omega_0}\left[\left\langle A_k\nabla u_k,\nabla \varphi\right\rangle+f_k\varphi\right]dx=0\end{equation}
 and so passing to the limit 
 \begin{equation}\label{comp2}\int\limits_{\Omega_0}\left[\left\langle \underbrace{(\det \Theta')(\Theta')^{-1}(^t\Theta')^{-1}}_{=A}\nabla u,\nabla \varphi\right\rangle-\underbrace{\(| u|^{p-1}  u+\lambda u\)(\det\Theta'}_{=f})\varphi\right]dx=0,\end{equation}
namely
$${\mathrm {div}}\(A\nabla u\)+f=0\ \hbox{in}\ \Omega_0,\ u=0\ \hbox{on}\ \partial\Omega_0,$$
i.e. $F(\lambda,u,\theta)=0.$\\
Now, let us prove that $u_k\to u$ strongly in $H^1_0(\Omega_0)\cap H^2(\Omega_0).$ 
By \eqref{comp1}  and \eqref{comp2}   we deduce
$$\begin{aligned}\int\limits_{\Omega_0} \left\langle A\nabla (u_k-u),\nabla (u_k-u)\right\rangle&=\int\limits_{\Omega_0} \left\langle  A\nabla u_k, \nabla u_k\right\rangle+ \int\limits_{\Omega_0} \left\langle  A\nabla u , \nabla u \right\rangle-2\int\limits_{\Omega_0} \left\langle  A\nabla u , \nabla u_k\right\rangle\\
&=\int\limits_{\Omega_0} \left\langle  (A-A_k)\nabla u_k, \nabla u_k\right\rangle+ \int\limits_{\Omega_0}\(-f_ku_k-fu+2fu_k\)\\
&=o(1), 
\end{aligned}$$
because $A_k\to A$ in $C^0(\Omega_0)$ and $u_k\to u$ strongly in $L^{2n\over n-2}(\Omega_0).$  
Moreover, we also have
$$\begin{aligned}\int\limits_{\Omega_0}\({\mathrm {div}}\(A\nabla (u_k-u)\)\)^2&=\int\limits_{\Omega_0}\({\mathrm {div}}\((A-A_k)\nabla  u_k\)-f_k+f \)^2\\
&\le 2\int\limits_{\Omega_0}\({\mathrm {div}}\((A-A_k)\nabla  u_k\)\)^2+2\int\limits_{\Omega_0}\(f_k-f \)^2\\
&=o(1), \end{aligned}$$
because $A_k\to A$ in $C^0(\Omega_0)$ and $u_k\to u$ strongly in $L^{2(n+2)\over n-2}(\Omega_0).$  
Then the claim follows directly from Remark \ref{norm}.
\end{proof}

 Let us check assumption ii) of Theorem \ref{tran}. 
 \begin{proposition}$(0,0)$ is a regular value of $F.$

 \end{proposition}
 \begin{proof}
 Let $(\lambda_0,u_0,\theta_0)\in U\times V$ such that $F(\lambda_0,u_0,\theta_0)=(0,0).$ We shall prove that 
 if $(\lambda,u)\in X$ is such that
 \begin{equation}\label{f5}
\left\{ \begin{aligned}&F'(\lambda_0,u_0,\theta_0)[\lambda,u]=0\\
 &\langle F'(\lambda_0,u_0,\theta_0)[\theta],(\lambda,u)\rangle_Z=0\ \hbox{for any}\ \theta\in Y\end{aligned}\right.\ \Rightarrow
\ \lambda=0\ \hbox{and}\ u\equiv0. \end{equation}
Without loss of generality we can assume $\theta_0=0.$ Then  $\Theta_0=I$ and by \eqref{f1} and \eqref{f2} condition $F(\lambda_0,u_0,\theta_0)=(0,0)$ reads as  
\begin{equation}\label{f6}
  \left\{\begin{aligned}&\int\limits_{\Omega_0}\(|\nabla u_0|^2-|u_0|^{p+1}-\lambda_0  u_0^2\) d x=0\\
&\Delta   u_0+|u_0|^{p-1}u_0+\lambda_0u_0=0\ \hbox{in}\ \Omega_0,\ u=0\ \hbox{on}\ \partial\Omega_0.
 \end{aligned}\right.\end{equation}
 Moreover by \eqref{f3} and \eqref{f4} condition \eqref{f5} can be rephrased as
\begin{equation}\label{f7}\left\{
\begin{aligned}
 &
   \int\limits_{\Omega_0}\left\{2\nabla u_0 \nabla u -\((p+1)|u_0|^{p-1}u_0+2\lambda_0  u_0\)u  -\lambda  u_0^2 \right\}d x=0\\
   & \Delta u+\(p| u_0|^{p-1}  +\lambda_0  \)u +\lambda u_0 =0\ \hbox{in}\ \Omega_0,\ u=0\ \hbox{on}\ \partial\Omega_0
   \end{aligned}\right.\end{equation}
  and
  \begin{equation}\label{f8}
  \begin{aligned}
 &\lambda \int\limits_{\Omega_0}\left\{\nabla u_0\cdot\left[(\mathrm {div}\theta)\nabla u_0-(\theta'+{}^t\theta')\nabla u_0\right]-\(|u_0|^{p+1}+ \lambda _0 u_0^2\)(\mathrm {div}\theta)\right\}d x\\
&+\int\limits_{\Omega_0}\left\{{\mathrm {div}}
\left[(\mathrm {div}\theta)\nabla u_0-(\theta'+{}^t\theta')\nabla u_0\right]+\(| u_0|^{p-1} u_0 +\lambda_0u_0  \)(\mathrm {div}\theta)\right\}udx=0 \ \forall\ \theta\in Y.
\end{aligned}  
  \end{equation} 
 We can simplify expression \eqref{f8}. Indeed, taking into account that
\begin{equation}\label{f10}\Delta u_0+\underbrace{| u_0|^{p-1} u_0 +\lambda_0u_0}_{=g(u_0)}=0\ \hbox{in}\ \Omega_0,\ u=0\ \hbox{on}\ \partial\Omega_0,
  \end{equation}
 we have 
  $$\begin{aligned}\mathrm {div}\left[(\mathrm {div}\theta)\nabla u_0-(\theta'+{}^t\theta')\nabla u_0\right]&=\mathrm {div}(\theta\Delta u_0)-\Delta (\theta\nabla u_0)=-\mathrm {div}(g(u_0)\theta)-\Delta (\theta\nabla u_0)\\ &=-g(u_0)(\mathrm {div}\theta)-g'(u_0)\nabla u_0\theta-\Delta (\theta\nabla u_0).\end{aligned}$$
Moreover,
  $$\int\limits_{\Omega_0}\Delta (\theta\nabla u_0)udx=-\int\limits_{\partial\Omega_0} \theta\nabla u_0\partial_\nu udx+\int\limits_{\Omega_0} \theta\nabla u_0\Delta udx$$
Therefore, \eqref{f8} reads as
\begin{equation}\label{f9}
  \begin{aligned}
0= &\lambda \int\limits_{\Omega_0}\left\{\left[g(u_0) u_0(\mathrm {div}\theta)+g'(u_0)u_0\nabla u_0\theta+\theta\nabla u_0\underbrace{\Delta u_0}_{=-g(u_0)}\right]-\underbrace{\(|u_0|^{p+1}+ \lambda _0 u_0^2\)}_{=g(u_0)u_0}(\mathrm {div}\theta)\right\}d x\\ &-\lambda \int\limits_{\partial\Omega_0}\theta\nabla u_0\partial_\nu u_0dx\\
&+\int\limits_{\Omega_0}\left\{\left[-g(u_0)u(\mathrm {div}\theta)-g'(u_0)u\nabla u_0\theta-\theta\nabla u_0\underbrace{\Delta u}_{=-g'(u_0)u-\lambda u_0}\right]+\underbrace{\(| u_0|^{p-1} u_0 +\lambda_0u_0  \)}_{=g(u_0)}(\mathrm {div}\theta)u\right\} dx\\ &+  \int\limits_{\partial\Omega_0}\theta\nabla u_0\partial_\nu udx\\ &
=\lambda \int\limits_{\Omega_0}\underbrace{\(g'(u_0)u_0-g(u_0)
+u_0\)}_{= (p-1)|u_0|^{p-1}u_0+u_0}\theta\nabla u_0dx +  \int\limits_{\partial\Omega_0}\theta\nabla u_0\(\partial_\nu u
-\lambda \partial_\nu u_0\)dx.
\end{aligned}  
  \end{equation} 

Now, we prove that $\lambda=0$. Indeed by taking deformations $\theta$ which take fix the boundary of $\Omega_0$ by \eqref{f9} we get
$$\lambda \int\limits_{\Omega_0} \left[ (p-1)|u_0|^{p-1}u_0+u_0\right]\theta\nabla u_0dx=0\ \hbox{for any}\ \theta\in V,\ \theta=0\ \hbox{on}\ \partial\Omega_0.$$
If $\lambda\not=0$ then we necessarily have
$$  u_0\left[ (p-1)|u_0|^{p-1}+1\right]\nabla u_0=0\ \hbox{a.e. in}\ \Omega_0,$$
and so $u_0\nabla u_0=0$ a.e. in $\Omega.$ This is not possible because $u_0$ solves \eqref{f10} and by the unique continuation theorem in \cite{AKS} we know
that $\mathrm {meas}\{x\in\Omega_0\ :\ u_0(x)=0\}=\mathrm {meas}\{x\in\Omega_0\ :\ \nabla u_0(x)=0\}=0.$\\

Since $\lambda=0$ by \eqref{f9} we deduce that
$$ \int\limits_{\partial\Omega_0}\theta\nabla u_0\partial_\nu u
dx =0\ \hbox{for any}\ \theta\in Y$$
and arguing exactly as in \cite{ST}, page 313-314, we deduce that $u=0.$ 
That concludes  the proof.
 \end{proof}

\begin{proposition}\label{least}
For any  $\theta\in \Xi$ as in \eqref{non-de} there exists $\lambda_\theta\in (0,\lambda_1(\Omega_\theta))$ such that
\begin{equation}\label{cru}\lambda_\theta=2\max\limits_{\Omega_\theta} u_{\lambda_\theta}.\end{equation}
\end{proposition}
\begin{proof} 
 Let  $\theta\in \Xi$ as in \eqref{non-de} and let us consider the perturbed domain $\Omega_\theta$. For any  $\lambda\in (0,\lambda_1(\Omega_\theta))$ let $u_{\lambda}$ be the least energy positive solution on the domain $\Omega_\theta$, which  is non-degenerate because of Theorem \ref{main}.
 Therefore, by the Implicit function Theorem we deduce that there exists a continuous curve $\lambda\to u_\lambda$
Let us consider the continuous function
$$f(\lambda):=\lambda-2\|u _\lambda\|_{L^{\infty}(\Omega_\theta)},\ \lambda\in (0,\lambda_1(\Omega_\theta)).$$
Since
$$\lim\limits_{\lambda\to0}\|u _\lambda|_{L^{\infty}(\Omega_\theta)}=+\infty\ \hbox{and}\ \lim\limits_{\lambda\to\lambda_1(\Omega_\theta)}\|u _\lambda\|_{L^{\infty}(\Omega_\theta)}=0$$
  (see \cite{Han91} and  the classical bifurcation theory, respectively),
there exists $\lambda_\theta$ such that $f(\lambda_\theta)=0$ and the claim \eqref{cru} follows.
\end{proof}
 
 \begin{proof}[\bf Proof of Theorem \ref{main-generico}.] It follows immediately by Theorem \ref{main} and Proposition \ref{least}.
 \end{proof}

\bibliographystyle{abbrv}
\bibliography{n=6_final}

\begin{thebibliography}{10}

\bibitem{aggpv}
A.~L. Amadori, F.~Gladiali, M.~Grossi, A.~Pistoia, and G.~Vaira.
\newblock A complete scenario on nodal radial solutions to the
  br\'ezis-nirenberg problem in low dimensions.
\newblock {\em preprint}.

\bibitem{AKS}
N.~Aronszajn, A.~Krzywicki, and J.~Szarski.
\newblock A unique continuation theorem for exterior differential forms on
  {R}iemannian manifolds.
\newblock {\em Ark. Mat.}, 4:417--453 (1962), 1962.

\bibitem{abp}
F.~V. Atkinson, H.~Brezis, and L.~A. Peletier.
\newblock Nodal solutions of elliptic equations with critical {S}obolev
  exponents.
\newblock {\em J. Differential Equations}, 85(1):151--170, 1990.

\bibitem{A}
T.~Aubin.
\newblock Espaces de {S}obolev sur les vari\'{e}t\'{e}s riemanniennes.
\newblock {\em Bull. Sci. Math. (2)}, 100(2):149--173, 1976.

\bibitem{B}
G.~Bianchi and H.~Egnell.
\newblock A note on the {S}obolev inequality.
\newblock {\em J. Funct. Anal.}, 100(1):18--24, 1991.

\bibitem{Brezis1983}
H.~Brezis and L.~Nirenberg.
\newblock Positive solutions of nonlinear elliptic equations involving critical
  sobolev exponents.
\newblock {\em Communications on Pure and Applied Mathematics}, 36(4):437--477,
  1983.

\bibitem{cgs}
L.~A. Caffarelli, B.~Gidas, and J.~Spruck.
\newblock Asymptotic symmetry and local behavior of semilinear elliptic
  equations with critical {S}obolev growth.
\newblock {\em Comm. Pure Appl. Math.}, 42(3):271--297, 1989.

\bibitem{cfp}
A.~Capozzi, D.~Fortunato, and G.~Palmieri.
\newblock An existence result for nonlinear elliptic problems involving
  critical {S}obolev exponent.
\newblock {\em Ann. Inst. H. Poincar\'{e} Anal. Non Lin\'{e}aire},
  2(6):463--470, 1985.

\bibitem{css}
G.~Cerami, S.~Solimini, and M.~Struwe.
\newblock Some existence results for superlinear elliptic boundary value
  problems involving critical exponents.
\newblock {\em J. Funct. Anal.}, 69(3):289--306, 1986.

\bibitem{ddm}
M.~del Pino, J.~Dolbeault, and M.~Musso.
\newblock The {B}rezis-{N}irenberg problem near criticality in dimension 3.
\newblock {\em J. Math. Pures Appl. (9)}, 83(12):1405--1456, 2004.

\bibitem{druet}
O.~Druet.
\newblock Elliptic equations with critical {S}obolev exponents in dimension 3.
\newblock {\em Ann. Inst. H. Poincar\'{e} Anal. Non Lin\'{e}aire},
  19(2):125--142, 2002.

\bibitem{epv}
P.~Esposito, A.~Pistoia, and J.~V\'{e}tois.
\newblock The effect of linear perturbations on the {Y}amabe problem.
\newblock {\em Math. Ann.}, 358(1-2):511--560, 2014.

\bibitem{Han91}
Z.~C. Han.
\newblock Asymptotic approach to singular solutions for nonlinear elliptic
  equations involving critical sobolev exponent.
\newblock {\em Annales de l'I.H.P. Analyse non linéaire}, 8(2):159--174, 1991.

\bibitem{iava1}
A.~Iacopetti and G.~Vaira.
\newblock Sign-changing tower of bubbles for the {B}rezis-{N}irenberg problem.
\newblock {\em Commun. Contemp. Math.}, 18(1):1550036, 53, 2016.

\bibitem{iava2}
A.~Iacopetti and G.~Vaira.
\newblock Sign-changing blowing-up solutions for the {B}rezis-{N}irenberg
  problem in dimensions four and five.
\newblock {\em Ann. Sc. Norm. Super. Pisa Cl. Sci. (5)}, 18(1):1--38, 2018.

\bibitem{mipive}
A.~M. Micheletti, A.~Pistoia, and J.~V\'{e}tois.
\newblock Blow-up solutions for asymptotically critical elliptic equations on
  {R}iemannian manifolds.
\newblock {\em Indiana Univ. Math. J.}, 58(4):1719--1746, 2009.

\bibitem{mupi}
M.~Musso and A.~Pistoia.
\newblock Multispike solutions for a nonlinear elliptic problem involving the
  critical {S}obolev exponent.
\newblock {\em Indiana Univ. Math. J.}, 51(3):541--579, 2002.

\bibitem{ms}
M.~Musso and D.~Salazar.
\newblock Multispike solutions for the {B}rezis-{N}irenberg problem in
  dimension three.
\newblock {\em J. Differential Equations}, 264(11):6663--6709, 2018.

\bibitem{P}
A.~Pistoia.
\newblock A generic property of the resonance set of an elliptic operator with
  respect to the domain.
\newblock {\em Proc. Roy. Soc. Edinburgh Sect. A}, 127(6):1301--1310, 1997.

\bibitem{pre}
B.~Premoselli.
\newblock Towers of bubbles for yamabe-type equations and for the
  br\'ezis-nirenberg problem in dimensions $n\ge7$.
\newblock {\em https://arxiv.org/abs/2009.01515}.

\bibitem{Q}
F.~Quinn.
\newblock Transversal approximation on {B}anach manifolds.
\newblock In {\em Global {A}nalysis ({P}roc. {S}ympos. {P}ure {M}ath., {V}ol.
  {XV}, {B}erkeley, {C}alif., 1968)}, pages 213--222. Amer. Math. Soc.,
  Providence, R.I., 1970.

\bibitem{rey}
O.~Rey.
\newblock The role of the {G}reen's function in a nonlinear elliptic equation
  involving the critical {S}obolev exponent.
\newblock {\em J. Funct. Anal.}, 89(1):1--52, 1990.

\bibitem{rv}
F.~Robert and J.~V\'{e}tois.
\newblock Sign-changing blow-up for scalar curvature type equations.
\newblock {\em Comm. Partial Differential Equations}, 38(8):1437--1465, 2013.

\bibitem{ST}
J.-C. Saut and R.~Temam.
\newblock Generic properties of nonlinear boundary value problems.
\newblock {\em Comm. Partial Differential Equations}, 4(3):293--319, 1979.

\bibitem{Sri93}
P.~N. Srikanth.
\newblock Uniqueness of solutions of nonlinear dirichlet problems.
\newblock {\em Differential Integral Equations}, 6(3):663--670, 1993.

\bibitem{T}
G.~Talenti.
\newblock Best constant in {S}obolev inequality.
\newblock {\em Ann. Mat. Pura Appl. (4)}, 110:353--372, 1976.

\bibitem{U}
K.~Uhlenbeck.
\newblock Generic properties of eigenfunctions.
\newblock {\em Amer. J. Math.}, 98(4):1059--1078, 1976.

\bibitem{va}
G.~Vaira.
\newblock A new kind of blowing-up solutions for the {B}rezis-{N}irenberg
  problem.
\newblock {\em Calc. Var. Partial Differential Equations}, 52(1-2):389--422,
  2015.

\end{thebibliography}
 
\end{document}